\documentclass[11pt, reqno, twoside, makeidx]{amsart}
\usepackage[margin=2.7cm, marginpar=1cm]{geometry}

\usepackage{wrapfig}
\usepackage{marginnote}
\usepackage{hyperref}
\usepackage[english]{babel}
\usepackage[utf8]{inputenc}
\usepackage[T1]{fontenc}
\usepackage{amsmath,amsthm,amssymb,amsfonts}
\numberwithin{equation}{section}
\usepackage{enumerate}
\usepackage{enumitem}
\usepackage{faktor}
\usepackage{dsfont}
\usepackage{nicematrix}
\usepackage{scalerel,stackengine}
\usepackage{textcomp}
\usepackage{graphicx}
\usepackage{extarrows}
\usepackage{epigraph}
\usepackage{graphicx}
\usepackage{subcaption}
\usepackage{sidecap}
\usepackage{indentfirst}
\usepackage{tikz}
\usepackage{tikz-cd}
\usepackage{tkz-euclide}
\usetikzlibrary{shapes.misc,arrows,matrix,patterns,decorations.markings,positioning}
\tikzset{cross/.style={cross out, draw=black, minimum size=2*(#1-\pgflinewidth), inner sep=0pt, outer sep=0pt},
cross/.default={4.5pt}}
\usepackage{pgfplots}
\usepackage{caption}
\usepackage{float}
\usepackage{color}
\usepackage{epstopdf}
\usepackage{epsfig}
\usepackage{stmaryrd}
\usepackage{color}
\usepackage{geometry}
\usepackage{multicol}
\usepackage{verbatim}
\usepackage{wrapfig}
\usepackage{float}

\DeclareMathOperator{\tb}{tb}

\DeclareMathOperator{\Spin}{Spin}

\DeclareMathOperator{\xist}{\xi_{std}}
\renewcommand{\geq}{\geqslant}
\renewcommand{\leq}{\leqslant} 
\renewcommand{\epsilon}{\varepsilon}
\newcommand{\R}{\mathbb{R}}

\newcommand{\Z}{\mathbb{Z}}
\newcommand{\Q}{\mathbb{Q}}
\newcommand{\F}{\mathbb{F}}

\newcommand{\C}{\mathbb C}

\newcommand{\s}{\mathfrak{s}}

\DeclareFontFamily{U}{mathx}{\hyphenchar\font45}
\DeclareFontShape{U}{mathx}{m}{n}{
      <5> <6> <7> <8> <9> <10>
      <10.95> <12> <14.4> <17.28> <20.74> <24.88>
      mathx10
      }{}
\DeclareSymbolFont{mathx}{U}{mathx}{m}{n}
\DeclareFontSubstitution{U}{mathx}{m}{n}
\DeclareMathAccent{\widecheck}{0}{mathx}{"71}
\DeclareMathAccent{\wideparen}{0}{mathx}{"75}

\newtheorem{teo}{Theorem}[section]
\newtheorem*{teo*}{Theorem}
\newtheorem{lemma}[teo]{Lemma}
\newtheorem{prop}[teo]{Proposition}
\newtheorem*{prop*}{Proposition}

\newtheorem{conj}[teo]{Conjecture}

\newtheorem{remark}[teo]{Remark}

\usepackage{xpatch}
\makeatletter
\xpatchcmd{\@thm}{\thm@headpunct{.}}{\thm@headpunct{}}{}{}
\makeatother

\pgfplotsset{compat=1.18}
\begin{document}
\title[Brieskorn spheres and rational homology ball symplectic fillings]{Brieskorn spheres and rational homology ball \\ symplectic fillings}
\author{Antonio Alfieri, Alberto Cavallo and Irena Matkovi\v c}
\dedicatory{\smaller[1] To Andr\'as I. Stipsicz on the occasion of his 60th birthday}

\begin{abstract}
 Given a canonically oriented Brieskorn sphere $Y=\Sigma(a_1,...,a_n)$, we confirm some statements conjectured by Gompf. More specifically, we obstruct the existence of rational homology ball symplectic fillings for any contact structure on $-Y$ if $n=3$, and when there is no half convex Giroux torsion for $n>3$. Furthermore, we show that the same result holds for the Milnor fillable structure on $Y$ with the possible exception of $\Sigma(3,4,5),$ $\Sigma(2,5,7)$ and $\Sigma(2,3,6k+1)$ for $k\geq1$. Along the way, we determine every canonically oriented Brieskorn sphere with vanishing correction term carrying at most two fillable structures, up to isotopy. 
\end{abstract}

\maketitle

\thispagestyle{empty}
\vspace{-0.5cm}
\section{Introduction}
A \emph{Stein domain} is a compact complex manifold with a plurisubharmonic exhaustion function, whose pseudo-convex boundary has a naturally induced contact structure. From the point of view of a 3-manifold, contact structures arising this way are called \emph{Stein fillable}. They are necessarily \emph{tight}, and inherit rigidity from the complex structure through relative adjunction inequalities \cite{Eliashberg}. In addition, the topology of Stein fillings has deep connections with the problem of symplectic embeddability, and  with the study of the support genus of contact structures in the sense of Giroux’s correspondence; see \cite{Li-Mak,GGP} and references therein.
 
A 4-manifold with boundary $X$ is called a \emph{homology ball} if it is indistinguishable from the standard $4$-ball when regarded under the lens of singular homology, that is $H_*(W;\Z)\simeq H_*(D^4;\Z)$. Similarly, we say that $X$ is a \emph{rational homology ball} if $H_*(W;\Q)\simeq H_*(D^4;\Q)$. Stein rational homology balls, in particular, are interesting because they provide symplectic manifolds with small Betti numbers. Such examples are in turn used to show the existence of exotic structures on  $\C P^2 \# m \overline{\C P^2}$ for small $m$, as extensively investigated in a series of papers by Park, Stipsicz and Szab\'o, culminating in \cite{PSSz}. Note that due to pseudo-convexity, Stein rational homology balls have simple handle decomposition, only containing $1$- and $2$-handles, and tend to be rare. Many of them can be produced as double branched cover of holomorphic disks properly embedded in $D^4$, as it was shown by Loi and Piergallini \cite{LP}; see \cite{Rama,BS} for different constructions.

In this paper we will concentrate on the 3-dimensional \emph{Brieskorn spheres}, because of their prominent role in the study of complex surface singularities. By definition, the Brieskorn homology sphere $\Sigma(a_1,...,a_n)$ is the link of a complete intersection singularity described as follows. Suppose that the coefficients $a_1,...,a_n$ are pairwise coprime, and choose a collection of complex coefficients $c_{i,j}$ where $1\leq i\leq n$ and $1\leq j\leq n-2$. Then consider the algebraic surface 
\[V(a_1,...,a_n)=\left\{(z_1,...,z_n)\in\C^n\:|\: c_{1,j}z_1^{a_1}+\cdots+c_{n,j}z_n^{a_n}=0, \:j=1,...,n-2\right\}\,.\] For a generic choice of the coefficients, so that the $n\times(n-2)$-matrix $(c_{i,j})$ has non-vanishing maximal minors, this variety has an isolated singularity at the origin, and the Brieskorn sphere arises as \[\Sigma(a_1,...,a_n) = S_\epsilon^{2n-1}(0)\cap V(a_1, \dots, a_n)\] for some small enough positive $\epsilon$.

Beyond their importance in singularity theory, Brieskorn spheres are deeply intertwined with Stein geometry and contact topology. They were central for the advancement of 3-dimensional contact topology: Gompf studied them in his seminal paper \cite{Gompf2}; moreover, they were used to produce examples of manifolds without symplectically fillable structures \cite{Lisca} or any tight contact structures \cite{EH}, pairs of non-isotopic tight structures which are homotopic as plane bundles \cite{LM}, and contact structures that are symplectically but not Stein fillable \cite{G-notStein}. 

We focus on two problems, deciding which Brieskorn spheres bound (smooth) rational homology balls, and determining which of them carry a Stein structure.
Our research in specific was motivated by the work of Gompf, where the following conjecture \cite[Conjecture 3.3]{Gompf} was formulated.

\begin{conj}[Gompf]\label{gompfconj} No non trivial Brieskorn homology sphere, with either orientation, arises as the boundary of a pseudo-convex domain in $\C^2$.
\end{conj}

A counter-example to Conjecture \ref{gompfconj} would necessarily be a Brieskorn sphere $\Sigma(a_1,...,a_n)$ which smoothly embeds in $S^4$. This is a very restrictive assumption, constrained for example by Donaldson's diagonalisation theorem, 
see \cite{AmCP} for a survey about this topic. 

Note that there are many examples of Brieskorn spheres bounding smooth homology balls \cite{Savk}, and even more bounding rational homology balls; these appear in recent papers by Akbulut and Larson \cite{AL}, and \c{S}avk \cite{Savk}. Only the first ones are candidates for having a rational homology ball Stein filling, the argument is purely homological and we briefly describe it below. For this reason none of the Brieskorn spheres in the family $\Sigma(2,3,6k+1)$ with $k$ odd can be the boundary of a pseudo-convex domain in $\C^2$.

\begin{lemma}
 \label{lemma:Stein_ball}
 Any rational homology ball Stein filling $W$ of an integral homology sphere $Y$ is a homology ball, that is $H_1(W; \Z)\simeq H_2(W;\Z)=0$.
\end{lemma}
\begin{proof} 
 From the exact sequence of the pair $(W,Y)$  we can conclude that $H_1(W;\Z)\simeq H_1(W,Y; \Z)$. By Poincar\' e duality and the universal coefficient theorem on the other hand: 
 \[H_1(W,Y; \Z)\simeq H^3(W;\Z)\simeq \text{Hom}(H_3(W; \Z),\Z) \oplus \text{Ext}(H_2(W; \Z), \Z) \ . \]
 Finally, note that since a Stein domain has the homotopy type of a $2$-complex one has that $H_3(W;\Z)=0$, but also $H_2(W;\Z)=0$ because $W$ is a rational homology ball.  
\end{proof}

\begin{remark}
 We recall that any Seifert fibred space $M(e_0;r_1,...,r_n)$ with vanishing $b_1$ can be presented by a negative-definite $n$-leg star-shaped graph, with central framing $e_0\in\Z$ and other vertices with framing at most $-2$, with exactly one orientation. Every canonically oriented Brieskorn sphere is of this kind, see \cite{Saveliev}. We call these Seifert fibred spaces \emph{negative-definite}, while we call the ones with opposite orientation \emph{indefinite}.   
\end{remark}

Issa and McCoy show \cite[Corollary 3]{ImC-e} that if $\Sigma(a_1,...,a_n)$, with the canonical orientation, bounds a rational homology ball then its \emph{standard graph} above has $e_0=-1$. Later on, they also prove \cite[Theorem 1.1]{ImC-embedding} that in general if a negative-definite Seifert fibred space is smoothly embeddable in $S^4$ then $e_0\geq-\frac{n+1}{2}$; furthermore, equality holds only for the manifold \begin{equation}M\left(-\frac{n+1}{2};1-\frac{1}{a},\frac{1}{a},1-\frac{1}{a},\frac{1}{a},...,1-\frac{1}{a}\right)\hspace{0.5cm}\text{ with }\hspace{0.5cm}a\geq2\hspace{0.5cm}\text{ integer.}\label{eq:Duncan}\end{equation}

We use their results to answer Conjecture \ref{gompfconj} in the affirmative for oppositely oriented Brieskorn spheres in absence of half convex Giroux torsion; indeed, we are able to obstruct the existence of any rational homology ball symplectic filling of $-\Sigma(a_1,...,a_n)$, a much stronger result. Note that, remarkably, bounding a pseudo-convex domain in $\C^2$ is equivalent to having a homology ball Stein filling which smoothly embeds in $\C^2$ from a result of Gompf \cite[Corollary 3.1]{Gompf}.

When $n=3$ the presence of half convex torsion always produces an overtwisted structure; therefore, in this case our obstruction for $-Y$ is complete and strengthens \cite[Theorem 1.14]{EOT}. A half torsion layer always implies that $\xi$ is zero-twisting; by \cite[Proposition 1.7]{CM} this means that $c^+(\xi)$ is vanishing in $HF_\text{red}(Y)$. It is expected \cite{MV} that, for Seifert fibred spaces, half Giroux torsion causes the contact invariant to vanish.

\begin{teo}
 \label{teo:canonical1}
 Let $Y=\Sigma(a_1,a_2,a_3)$ be a canonically oriented Brieskorn sphere; then $-Y$ has no rational homology ball symplectic filling. Furthermore, the same is true for $(Y,\xi)$ when $Y$ is different from $\Sigma(3,4,5),$ $\Sigma(2,5,7)$ and $\Sigma(2,3,6k+1)$ for $k\geq1$, and $\xi$ is either $\xi_{\emph{can}}$, the Milnor fillable structure on $Y$, or its conjugate. 
\end{teo}

It also needs to be remarked that we do not know any example of Brieskorn spheres with $n \geq 4$ exceptional fibres bounding a homology ball. The lack of examples motivated a conjecture of Koll\'ar stated in his influential survey article \cite[Conjecture 20]{Kollar}.

For the Milnor fillable contact structure $\xi_\text{can}$, an obstruction to bound a rational homology ball symplectic filling has already been given, in the case of $(\Sigma(p,q,pqk+1),\xi_\text{can})$ with $(p,q)\neq(2,3)$ and $k\geq1$, by a direct computation of $d_3(\xi_\text{can})$, see \cite[Theorem 1.16]{EOT}. Note that $\Sigma(p,q,pqk+1)\simeq S^3_{-1/k}(T_{p,q})$ where $1<p<q$ coprime. 

We also want to emphasise that, on any negative-definite Seifert fibred space, the only examples of contact structures which are known to admit a rational homology ball symplectic filling are the Milnor fillable structure $\xi_\text{can}$ and its conjugate $\overline\xi_\text{can}$, see \cite{BS,EOT}.

\begin{teo}
 \label{teo:canonical2}
 Let $Y=\Sigma(a_1,...,a_n)$ with $n>3$ be a canonically oriented Brieskorn sphere; then $(-Y,\eta)$ has no rational homology ball symplectic filling for any contact structure $\eta$ without half convex Giroux torsion on $-Y$. The same is true for $(Y,\xi)$ when $\xi$ is either $\xi_{\emph{can}}$ or its conjugate. 
\end{teo}

If a contact 3-manifold admits a symplectic filling which is a rational homology ball, then we know from Heegaard Floer theory that it has $d_3$-invariant and correction term $d$ equal to zero, while from \cite[Corollary 3]{ImC-e} we know that if it is Seifert fibred then its negative-definite standard graph has $e_0=-1$. To prove the second part of Theorems \ref{teo:canonical1} and \ref{teo:canonical2} we show first that $d_3(\xi_\text{can})$ is strictly minimal among the $d_3$-invariants of any fillable contact structure on a canonically oriented Brieskorn sphere; since $d$ is an upper bound for the possible $d_3$-invariants, this implies that the filling can exist only when there are at most two fillable contact structures, and then study which manifolds satisfy these assumptions. We proved in \cite[Theorem 1.9]{CM} that there is a unique structure only in the case of $\Sigma(2,3,5)$, but this manifold does not bound a rational homology ball.

We recall our classification of fillable contact structures on negative-definite Seifert fibred spaces from \cite[Theorem 1.1]{CM-negative}.

\begin{teo}[Cavallo-Matkovi\v c]
 \label{teo:CM-negative}
 Suppose that $M(e_0;r_1,\dots,r_n)$ has negative-definite standard graph. Then the symplectically fillable contact structures on $M$ are precisely the Legendrian surgeries on all possible Legendrian realisations of the complete blow-down of the standard graph.
\end{teo}

We use Theorem \ref{teo:CM-negative} together with some elementary computations to determine the canonically oriented Brieskorn spheres $Y$ such that $(Y,\xi_\text{can})$ may have a rational homology ball symplectic filling, according to the restrictions above.

\begin{prop}
 \label{prop:two}
 The only canonically oriented Brieskorn spheres which carry two fillable structures, up to isotopy, and have vanishing correction term are $\Sigma(3,4,5),$ $\Sigma(2,5,7)$ and $\Sigma(2,3,6k+1)$ for $k\geq1$.
\end{prop}

Note that from convex surface theory we have that $\xi_\text{can}$ and $\overline\xi_\text{can}$ are actually the only tight structures on each of the Brieskorn spheres in Proposition \ref{prop:two}. 

The strategy to prove the first part of Theorems \ref{teo:canonical1} and \ref{teo:canonical2} requires our classification of all the relevant contact structures on indefinite Seifert fibred spaces whose base orbifold is a sphere \cite{CM}; in particular, we need the following result, see \cite[Proposition 1.7]{CM}.

\begin{prop}[Cavallo-Matkovi\v c]
 \label{prop:CM-indefinite}   
 Suppose that $M=M(e_0;r_1,\dots,r_n)$ has indefinite standard graph. Then $\eta$ is a negative-twisting tight structure on $M$ if and only if $c^+(\eta)\in HF_\emph{red}(-M,\s_\eta)$ is non-vanishing.
\end{prop}

Using the same strategy, we can extend our obstructions to any indefinite Seifert fibred space, and determine exactly which ones bound a rational homology ball that embeds in $S^4$ and carry a Stein structure.

\begin{teo}
 \label{teo:opposite}
 Let $M=M(e_0;r_1,...,r_n)$ be a negative-definite Seifert fibred space with $n\geq3$; then $-M$ arises as the boundary of a pseudo-convex domain in $\C^2$, without half convex Giroux torsion on the boundary, if and only if one has $M=M(-2;1-\frac{1}{a},\frac{1}{a},1-\frac{1}{a})$ for $a\geq2$ integer. 
\end{teo}

\subsection*{Acknowledgements} {\smaller[1] We are happy to thank Georgios Dimitroglou Rizell for the many discussions we had, and Marco Golla for finding a mistake in an early version of the paper. We thank the Matematiska institutionen at Uppsala universitet for their friendly hospitality. A.C. has been partially supported by the HORIZON-ERC-2023-ADG 101141468 KnotSurf4d project.}

\section{Rational homology ball symplectic fillings}
We denote by $Q_{X_\Gamma}$ the intersection form of $X_\Gamma$, the negative-definite Stein domain obtained by blowing down the standard graph $\Gamma$ of $Y$. We prove the following linear algebra lemma whose proof is based on \cite[Proposition 4.1]{LySi} and which appears in a weaker form in \cite[Lemma 3.2]{EOT}.

We recall that a symmetric square matrix $A$ of size $m>1$ is \emph{irreducible} when $P^TAP$ is not a block matrix for any permutation matrix $P$. Such a terminology is standard in linear algebra, but we prefer to state it explicitly for the sake of clarification.

\begin{lemma}
 \label{lemma:EOT}
 Consider the quadratic form $F:\R^m\rightarrow \R$ given as $F(V)=V^TAV$ where $A$ is the inverse of an irreducible negative-definite matrix $Q$ such that $q_{ij}\geq0$ for $1\leq i\neq j\leq m$, and let $W$ be a fixed vector in $\R^m$ whose entries are non-negative. If we set \[\mathcal D_W=\{V\in\R^m\:|\:|v_i|\leq w_i\text{ for every }1\leq i\leq m\}\:,\]
then $F\lvert_{\mathcal D_{W}}$ attains its minimum at $W$. Furthermore, if $F(W)=F(V)$ then one has $V=\pm W$. 
\end{lemma}
\begin{proof}
 We have that every entry of $A$ is negative because $Q$ is an irreducible negative-definite $M$-matrix, see \cite[Theorem 6.2.7]{BP}. 
 If $V\in\mathcal D_W$ then $w_i\pm v_i\geq0$ for every $i=1,...,m$. This immediately implies \[0\geq(W+V)^TA(W-V)=F(W)-F(V)\] as $W\pm V$ is a vector with non-negative coordinates, and $A$ has negative entries. Moreover, suppose that $V$ satisfies $F(V)=F(W)$; then we set $a_i=w_i+v_i,\:b_i=w_i-v_i$ and $C=A(W-V)$. We have $a_i,b_i\geq0$ while $c_i\leq0$ for $i=1,...,m$.

 Since by assumption one has $0=(W+V)^TA(W-V)$, we can write $0=\sum_{i=1}^ma_ic_i$ and then either $a_i=0$ or $c_i=0$ for each $i$; if $a_i=0$ for $i=1,...,m$ then $V=-W$. Suppose instead that there is an $i_0$ such that $a_{i_0}>0$; we have $0=c_{i_0}=\sum_{i=1}^ma_{i_0j}b_j$ and then $b_j=0$ for $j=1,...,m$. This implies $V=W$.
\end{proof}

We now suppose that $\xi$ is either $\xi_{\text{can}}$ or $\overline\xi_{\text{can}}$ and obstruct the existence of rational homology ball symplectic fillings.

\begin{prop}
 \label{prop:canonical}
 Let $Y=\Sigma(a_1,...,a_n)$ be a canonically oriented Brieskorn sphere, different from $\Sigma(3,4,5),$ $\Sigma(2,5,7)$ and $\Sigma(2,3,6k+1)$ for $k\geq1$, and $\xi_{\emph{can}}$ be the Milnor fillable structure on $Y$. Then $(Y,\xi_{\emph{can}})$ has no rational homology ball symplectic filling.
\end{prop}
\begin{proof}
 Since $Y$ bounds a rational homology ball carrying a symplectic structure, its invariant $d_3(\xi)$ and correction term $d=d(Y)$ both vanish.
 
 The result follows by showing that the only Brieskorn spheres which admit exactly two fillable contact structures, and have vanishing correction term and $d_3$-invariant, are $\Sigma(3,4,5),$ $\Sigma(2,5,7)$ and $\Sigma(2,3,6k+1)$ for $k\geq1$, as stated in Proposition \ref{prop:two}. In order to argue that $Y$ cannot have a unique structure, up to isotopy, we can look at \cite[Proposition 1.5]{CM-negative}: in fact, the structure $\xi$ would be self-conjugate, and then $d(Y)>0$ which is a contradiction; we also know from \cite[Theorem 1.9]{CM} that (with the canonical orientation) the only such manifold is $\Sigma(2,3,5)$, and its correction term is 2.

 To see that there are exactly two fillable structures on $Y$, we assume that another $\xi'$ exists. Since the 4-manifold $X_\Gamma$ defined above admits Stein structures $J$ and $J'$ inducing both $\xi$ and $\xi'$ from Theorem \ref{teo:CM-negative}, and they are determined by the rotation numbers of the Legendrian knots $\mathcal K_1,...,\mathcal K_m$ in the surgery presentation coming from blowing down $\Gamma$, we have that \[4d_3(\xi')-m=c_1^2(J')[X_\Gamma]=V^TQ_{X_\Gamma}^{-1}V\] for a certain characteristic vector $V$ whose coordinates satisfy $|v_i|\leq -q_{ii}+\text{TB}_{\xist}(\mathcal K_i)-1$ for $1\leq i\leq m$, and where the $q_{ii}$'s are the diagonal entries of $Q_{X_\Gamma}$. Then one has \[0=d_3(\xi)=d_3(\overline\xi_{\text{can}})\leq d_3(\xi')\leq d(Y)=0\:;\] the second inequality comes from the fact that, in a manifold presented by a negative-definite star-shaped graph, the correction term is the maximal Maslov grading of any non-zero homogeneous class in $\widehat{HF}$, see \cite[Proposition 2.2]{CM-negative}, while the first one comes from the existence part of Lemma \ref{lemma:EOT}. Hence, it follows from the uniqueness part of Lemma \ref{lemma:EOT} that $V=\pm V_{\text{can}}$.
 Note that we can apply Lemma \ref{lemma:EOT} because $Q_{X_\Gamma}$ is irreducible, as the link obtained by blowing down $\Gamma$ is non-split and positive.
 
 Together with \cite[Proposition 5.1]{CM-negative}, this tells us that there is no fillable structure on $Y$ other than $\xi_{\text{can}}$ and $\overline\xi_{\text{can}}$. In fact, the fillable structures on a negative-definite Seifert fibred space are distinguished by the $\Spin^c$-structure on $X_\Gamma$ (which is determined by the characteristic vector $V$) induced by the corresponding Stein structure.   
\end{proof}

Before proving our main results, we need an additional lemma.

\begin{lemma}
 \label{lemma:half}
 If $M=M(e_0;r_1,...,r_n)$ carries a zero-twisting tight contact structure without half convex Giroux torsion then $e_0\geq-1$.   
\end{lemma}
\begin{proof}
 We show that if a contact structure $\xi$ on $M$ with $e_0\leq -2$ is zero-twisting, it must have some half torsion. Assuming that there is a vertical fibre $K$ with twisting zero, we can thicken the standard neighbourhoods $\nu K_i$ of singular fibres, up to $\widetilde NK_i$ with $M\backslash (\cup_i\widetilde NK_i)$ having $\infty$ boundary slopes. Furthermore, since $e_0\leq -2$ we know that $M$ admits $-1$-twisting structure and in particular, there are neighbourhoods $N K_i$ such that $\nu K_i\subset N K_i\subset \widetilde NK_i$ and $M\backslash (\cup_i NK_i)$ has boundary slopes $-p_i=-1$ for $i=1,...,n-1$ and $-p_n=e_0+1$ for $i=n$. We notice that for all singular fibres the layer $\widetilde NK_i\backslash NK_i$ consists of a single basic slice and we enumerate legs so that the ones with sign opposite to the one on the last leg come first, say there are $m$ of them. We join $-\partial(M\backslash NK_i)$ for $i$ with positive, respectively negative, sign by vertical annuli avoiding zero-twisting fibre. Because of \cite[Lemma 4.13]{GSch} we obtain this way two neighbouring thickened tori with slopes $-\sum_{i=1}^m p_i+(m-1)=\sum_{i=m+1}^n p_i-(n-m-1)+e_0=-1$ and $\infty$, attached along $\infty$-slope, which form a half convex torsion layer. Therefore, the structure is either overtwisted when $m=0,1, n-1$, or it has half torsion along separating torus.   
\end{proof}

\begin{proof}[Proof of Theorems \ref{teo:canonical1} and \ref{teo:canonical2}]
 The second part, when $\xi$ is either $\xi_\text{can}$ or $\overline\xi_\text{can}$, is done in Proposition \ref{prop:canonical}. 
 
 For the first part, that is the case of $-Y$, we need Proposition \ref{prop:CM-indefinite}. Since there should be a structure $\eta$ such that $(-Y,\eta)$ has a rational homology ball symplectic filling, we have $d_3(\eta)=d(Y)=0$; hence, the invariant $c^+(\eta)\in HF^+(Y)$ is non-zero and has even parity, that is $M(c^+(\eta))\equiv d(Y)\text{ mod }2$. From Proposition \ref{prop:CM-indefinite} it follows that if $\eta$ is negative-twisting then $c^+(\eta)\in HF_\text{red}(Y,\s_\eta):=HF^+(Y,\s_\eta)/\mathcal T$ is non-zero, where $\mathcal T$ is the only subgroup of $HF^+(Y,\s_\eta)$ isomorphic to $\F[U,U^{-1}]/U\cdot\F[U,U^{-1}]$. Our \cite[Proposition 2.2]{CM-negative} combined with a result of Ozsv\'ath and Szab\'o \cite[Corollary 1.4]{OSz-fullpath} gives that the subgroup spanned by the non-zero classes of $\widehat{HF}(Y)$ with even parity is the image of $\psi_*:HF^-(Y)\rightarrow\widehat{HF}(Y)$; in particular, this implies that $\mathcal T$ consists of all the classes of $HF^+(Y)$ with even Maslov grading. In this way we have proved that $d_3(\eta)=-M(c^+(\eta))$ is necessarily odd, which gives a contradiction; thus $\eta$ should be zero-twisting.

 We know from Lemma \ref{lemma:half} that $-Y$ can have a zero twisting tight structure, without half convex Giroux torsion, only when its standard graph $\Gamma^*$, the dual of $\Gamma$, has $e_0\geq -1$. Note that we cannot have $e_0\geq0$ because $Y$ would be an $L$-space, implying $-Y=-\Sigma(2,3,5)$ which has no tight structure \cite{EH}. Using \cite[Corollary 3]{ImC-e}, we have that the (negative-definite) standard graph $\Gamma$ of $Y$ has also $e_0=-1$ which means $-1=-(-1)-n$, and thus $n=2$. This concludes the proof as $\Gamma$ has at least three legs: only lens spaces have a standard graph with two legs. 
\end{proof}

We now show that for many negative-definite Seifert fibred spaces as $M=M(e_0;r_1,...,r_n)$, the manifold $-M$ is not the boundary of a pseudo-convex domain in $\C^2$. Not every case can be obstructed as it is clear from Figure \ref{Pseudo} that there are examples of such manifolds; hence, we are going to show that these are the only ones.

\begin{figure}[ht]
 \centering
  \def\svgwidth{0.5\textwidth}
\begingroup%
  \makeatletter%
  \providecommand\color[2][]{%
    \errmessage{(Inkscape) Color is used for the text in Inkscape, but the package 'color.sty' is not loaded}%
    \renewcommand\color[2][]{}%
  }%
  \providecommand\transparent[1]{%
    \errmessage{(Inkscape) Transparency is used (non-zero) for the text in Inkscape, but the package 'transparent.sty' is not loaded}%
    \renewcommand\transparent[1]{}%
  }%
  \providecommand\rotatebox[2]{#2}%
  \newcommand*\fsize{\dimexpr\f@size pt\relax}%
  \newcommand*\lineheight[1]{\fontsize{\fsize}{#1\fsize}\selectfont}%
  \ifx\svgwidth\undefined%
    \setlength{\unitlength}{1505.7328448bp}%
    \ifx\svgscale\undefined%
      \relax%
    \else%
      \setlength{\unitlength}{\unitlength * \real{\svgscale}}%
    \fi%
  \else%
    \setlength{\unitlength}{\svgwidth}%
  \fi%
  \global\let\svgwidth\undefined%
  \global\let\svgscale\undefined%
  \makeatother%
  \begin{picture}(1,0.22109497)%
    \lineheight{1}%
    \setlength\tabcolsep{0pt}%
    \put(0,0){\includegraphics[width=\unitlength,page=1]{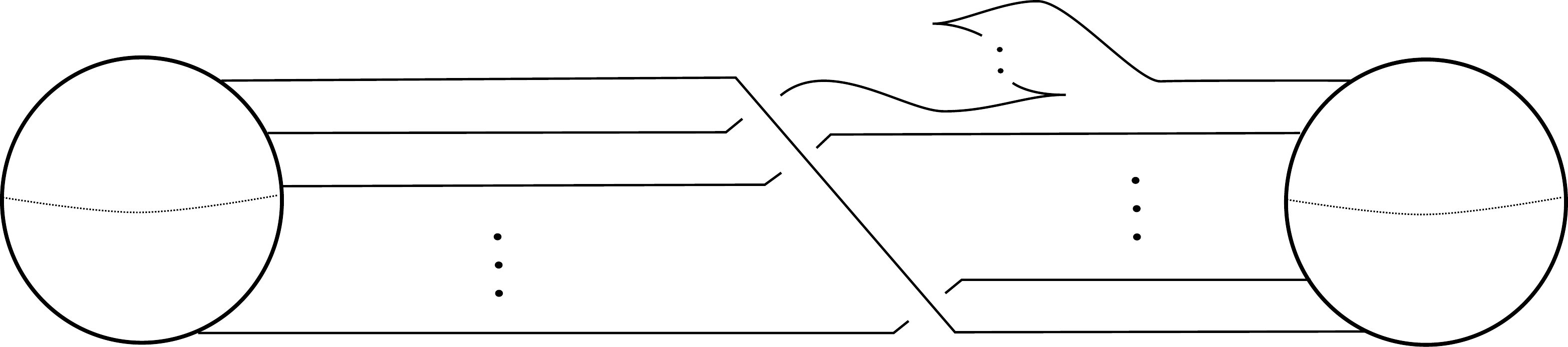}}%
    \put(0.34120722,0.20551395){\color[rgb]{0,0,0}\makebox(0,0)[lt]{\lineheight{1.25}\smash{\begin{tabular}[t]{l}$a-2\text{ times}$\end{tabular}}}}%
  \end{picture}%
\endgroup%

 \caption{\smaller[1]{A pseudo-convex domain in $\C^2$ whose boundary is $-M=M(-1;\frac{1}{a},1-\frac{1}{a},\frac{1}{a})\simeq S^3_{0,0}(T_{2,2a})$ for any integer $a\geq2$. The Legendrian knot where we attach the Stein 2-handle has $a$ strands, and its $\tb$-number is equal to 1.}}
 \label{Pseudo}
 \end{figure}

\begin{proof}[Proof of Theorem \ref{teo:opposite}]
 Let $-M$ be a Seifert fibred space whose standard graph $\Gamma^*$ is indefinite; in other words, whose plumbing has $b_2^+(P_{\Gamma^*})=1$. In the same way as in the proofs of Theorems \ref{teo:canonical1} and \ref{teo:canonical2} above, one has $d_3(\eta)=d(M,\s_\eta)=0$, and again $c^+(\eta)$ has even parity and then $\eta$ is zero-twisting. Now using \cite[Theorem 1.1]{ImC-embedding}, the standard graph $\Gamma$ of $M$ has $e_0\geq-\frac{n+1}{2}$; moreover, again from Lemma \ref{lemma:half} zero-twisting structures without half convex torsion on $-M$ can exist only when $\Gamma^*$ has $e_0\geq-1$. Combining these two facts, we obtain $-\frac{n+1}{2}\leq e_0\leq1-n$ which means $n\leq3$.

 From \cite[Theorem 1.1]{ImC-embedding} we also know that $M$ should be as in Equation \eqref{eq:Duncan}, because $M$ should bound a rational homology ball that smoothly embeds in $S^4$. Together with what we said above we obtain that $-M=M(-1;\frac{1}{a},1-\frac{1}{a},\frac{1}{a})$ for an integer $a\geq2$. All of these Seifert fibred spaces are the boundary of a pseudo-convex domain in $\C^2$ as shown in Figure \ref{Pseudo}.
\end{proof}

\section{Proof of Proposition \ref{prop:two}}
We prove that the only canonically oriented Brieskorn spheres $Y=\Sigma(a_1,...,a_n)$ with $d(Y)=0$, admitting exactly two symplectically fillable contact structures, are $\Sigma(3,4,5),$ $\Sigma(2,5,7)$ and the family $\Sigma(2,3,6k+1)$ for $k\geq1$, see Figure \ref{Family3}. Note that many of these manifolds bound smooth rational homology balls \cite{Savk}; moreover, by Theorem \ref{teo:CM-negative} the manifold $Y$ always has at least one fillable structure; namely, the structure $\xi_\text{can}$.

   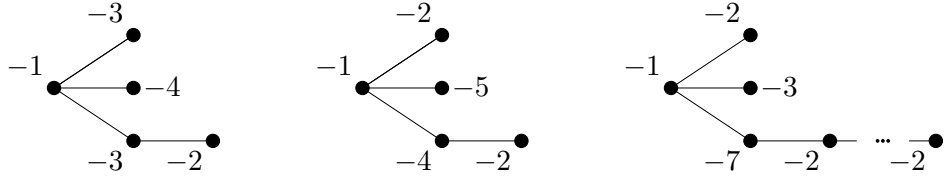
\begin{figure}[ht] 
     \begin{tikzpicture}[scale=0.7]
    \tkzDefPoints{0/0/A, 1.5/1/B, 1.5/-1/D, 3/-1/E, 1.5/0/C}  
    \tkzDrawSegment(A,B)\tkzDrawSegment(A,D)\tkzDrawSegment(E,D)\tkzDrawSegment(B,A)\tkzDrawSegment(A,C)
    \tkzDrawPoints[fill,black,size=5](A,B,C,D,E)
     \tkzLabelPoint[above left](A){$-1$} 
     \tkzLabelPoint[above left](B){$-3$}\tkzLabelPoint[right](C){$-4$}
     \tkzLabelPoint[below left](D){$-3$}\tkzLabelPoint[below left](E){$-2$}
       \end{tikzpicture}\hspace{1cm} 
       \begin{tikzpicture}[scale=0.7]
    \tkzDefPoints{0/0/A, 1.5/1/B, 1.5/-1/D, 3/-1/E, 1.5/0/C}  
    \tkzDrawSegment(A,B)\tkzDrawSegment(A,D)\tkzDrawSegment(E,D)\tkzDrawSegment(B,A)\tkzDrawSegment(A,C)
    \tkzDrawPoints[fill,black,size=5](A,B,C,D,E)
     \tkzLabelPoint[above left](A){$-1$} 
     \tkzLabelPoint[above left](B){$-2$}\tkzLabelPoint[right](C){$-5$}\tkzLabelPoint[below left](D){$-4$}\tkzLabelPoint[below left](E){$-2$}
       \end{tikzpicture}\hspace{1cm}
     \begin{tikzpicture}[scale=0.7]
    \tkzDefPoints{0/0/A, 1.5/1/B, 1.5/-1/D, 3/-1/E, 5/-1/F, 1.5/0/C}  
    \tkzDefPoints{3.5/-1/X, 4.5/-1/Y}\tkzDefPoints{3.9/-1/P, 4/-1/Q, 4.1/-1/R}
    \tkzDrawPoints[fill,black,size=1](P,Q,R)
    \tkzDrawSegment(A,B)\tkzDrawSegment(A,D)\tkzDrawSegment(E,D)\tkzDrawSegment(E,X)\tkzDrawSegment(Y,F)\tkzDrawSegment(A,C)
    \tkzDrawPoints[fill,black,size=5](A,B,C,D,E,F)
     \tkzLabelPoint[above left](A){$-1$} \tkzLabelPoint[above left](B){$-2$}\tkzLabelPoint[right](C){$-3$}
     \tkzLabelPoint[below left](D){$-7$}\tkzLabelPoint[below left](E){$-2$}\tkzLabelPoint[below left](F){$-2$}
\end{tikzpicture}
     \caption{\smaller[1]{The standard graph of $\Sigma(3,4,5)$ (left), $\Sigma(2,5,7)$ (middle) and $\Sigma(2,3,6k+1)$ with $k\geq1$ (right). There are $k-1$ vertices with framing $-2$ on the third leg of the graph of $\Sigma(2,3,6k+1)$.}}
     \label{Family3}
\end{figure} 

We reason as follows: the manifolds that we are looking for are precisely the ones such that the surgery presentation (obtained by blowing down the standard graph) has every knot $K$ in it appearing with framing $\text{TB}_{\xist}(K)-1$, except one with framing $\text{TB}_{\xist}(K)-2$; these numbers mean that all the Legendrian knots are not stabilised, except one that is stabilised once. We recall that every knot in the surgery presentation is an unknot, except when the blown down subgraph $\Gamma'\subset\Gamma$ represents the Seifert fibration of the torus knot $T_{d_2,d_1}$, where $1<d_2<d_1$ coprime. 

We recall that $e(Y):=r_1+\cdots+r_n+e_0$ where the $r_i$'s are the normalised coefficients of $Y$, and $e(Y)<0$ if and only if $Y$ is canonically oriented, see \cite{Saveliev}; in addition, we have the following result from \cite[Section 4]{CM-negative}. 

\begin{teo}[Cavallo-Matkovi\v c]
 \label{teo:tw}
 If $M=M(e_0;r_1,\dots,r_n)$ has negative-definite standard graph, then all negative-twisting tight contact structures on $M$ have the same maximal twisting number $\emph{tw}(M)$. Furthermore, the integer $\emph{tw}(M)$ is equal to $-d_1-d_2$ when the subgraph $\Gamma'$ represents the Seifert fibration of $T_{d_2,d_1}$ with $1\leq d_2\leq d_1$ coprime, or to $-1$ when $\Gamma'$ is empty, see Figure \ref{Subgraph}.
\end{teo}

\begin{figure}[ht]
\begin{tikzpicture}[scale=0.7]
    \tkzDefPoints{0/0/A}  \tkzDefPoints{0/-1.9/G} 
    \tkzDrawPoints[fill,black,size=5](A)\tkzDrawPoints[fill,white,size=1](G) 
     \tkzLabelPoint[below left](A){$-1$} 
 \end{tikzpicture}\hspace{1cm} 
 \begin{tikzpicture}[scale=0.7]
    \tkzDefPoints{1.5/-1/D, 3/-1/E, 5/-1/F} \tkzDefPoints{0/-2.87/G} 
    \tkzDefPoints{3.5/-1/X, 4.5/-1/Y}\tkzDefPoints{3.9/-1/P, 4/-1/Q, 4.1/-1/R}
    \tkzDrawPoints[fill,black,size=1](P,Q,R)\tkzDrawPoints[fill,white,size=1](G)
    \tkzDrawSegment(E,D)\tkzDrawSegment(E,X)\tkzDrawSegment(Y,F)
    \tkzDrawPoints[fill,black,size=5](D,E,F)
     \tkzLabelPoint[below left](D){$-1$}\tkzLabelPoint[below left](E){$-2$}\tkzLabelPoint[below left](F){$-2$}
\end{tikzpicture}\hspace{1.5cm}
\begin{tikzpicture}[scale=0.7]
    \tkzDefPoints{0/0/A, 1.5/1/B, 1.5/-1/D, 3/-1/E, 5/-1/F, 3/1/H, 5/1/I}  
    \tkzDefPoints{3.5/-1/X, 4.5/-1/Y}\tkzDefPoints{3.9/-1/P, 4/-1/Q, 4.1/-1/R}\tkzDefPoints{3.5/1/X', 4.5/1/Y'}\tkzDefPoints{3.9/1/P', 4/1/Q', 4.1/1/R'}
    \tkzDrawPoints[fill,black,size=1](P,Q,R)\tkzDrawPoints[fill,black,size=1](P',Q',R')
    \tkzDrawSegment(A,B)\tkzDrawSegment(A,D)\tkzDrawSegment(E,D)\tkzDrawSegment(E,X)\tkzDrawSegment(Y,F)\tkzDrawSegment(B,H)
    \tkzDrawSegment(H,X')\tkzDrawSegment(Y',I)
    \tkzDrawPoints[fill,black,size=5](A,B,D,E,F,H,I)
     \tkzLabelPoint[above left](A){$-1$} \tkzLabelPoint[above right](B){$m_1^1$}\tkzLabelPoint[above right](H){$m_2^1$}\tkzLabelPoint[above right](I){$m_{k_1}^1$}
     \tkzLabelPoint[below right](D){$m_1^2$}\tkzLabelPoint[below right](E){$m_2^2$}\tkzLabelPoint[below right](F){$m_{k_2}^2$}
\end{tikzpicture}
     \caption{\smaller[1]{The three different non-empty types for the subgraph $\Gamma'\subset\Gamma$: only the centre (left), corresponding to the fibration of $T_{1,1}$, the centre joint with the first $d_1-1$ vertices of one leg (middle), corresponding to the fibration of $T_{1,d_1}$, and the standard graph of $M(-1;\frac{b_1}{d_1},\frac{b_2}{d_2})$ where $b_i<d_i$ is the only positive integer such that $\frac{b_1}{d_1}+\frac{b_2}{d_2}=1-\frac{1}{d_1d_2}$ (right), corresponding to the fibration of $T_{d_2,d_1}$.}}
     \label{Subgraph}
\end{figure}
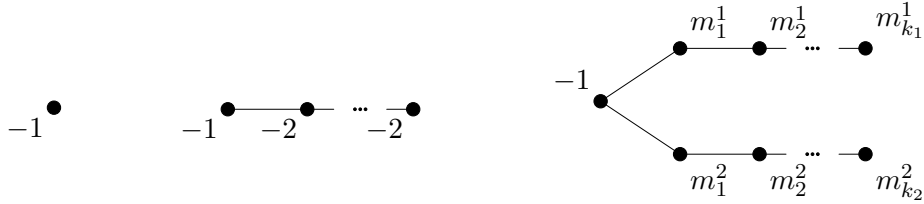

Since $d(Y)=0$ implies $e_0=-1$ by a result of Issa and McCoy \cite[Corollary 3]{ImC-e}, and thus the blown down subgraph $\Gamma'$ cannot be empty, we can distinguish three sub-cases:
  \begin{itemize}[leftmargin=0.5cm]
    \item $\text{tw}(Y)=-2$ \\
     We have that $e_0$ is $-1$ and it coincides with the subgraph $\Gamma'$; there is a single blow-down which increases by one the framing of every vertex connected to $e_0$. All the knots in the resulting surgery presentation are unknots. If $n>3$ then $e(Y)\geq (n-1)\cdot\frac{1}{3}+\frac{1}{4}-1>0$, thus $\Gamma$ would be indefinite; hence, we can assume that $n=3$. In addition, we observe that if there were no vertex with framing $-4$ connected to $e_0$ then $e(Y)\geq3\cdot\frac{1}{3}-1\geq0$, and again this is a contradiction. We have shown that $Y=M(-1;\frac{a+1}{2a+3},\frac{2b+1}{3b+4},\frac{c+1}{2c+3})$ where $a,b,c\geq0$ are the numbers of $-2$-vertices in each leg. 
     
     Now, if $a,c\geq1$ then $e(Y)\geq\frac{2}{5}+\frac{1}{4}+\frac{2}{5}-1=\frac{1}{20}>0$, thus $a=0$. If $c\geq2$ then $e(Y)\geq\frac{1}{3}+\frac{1}{4}+\frac{3}{7}-1=\frac{1}{84}>0$, thus $c=1$ because denominators need to be coprime ($Y$ is a Brieskorn sphere). Finally, if $b\geq1$ then $e(Y)\geq\frac{1}{3}+\frac{2b+1}{3b+4}+\frac{2}{5}-1=\frac{3b-1}{45b+60}>0$, implying $b=0$. This leaves $\Sigma(3,4,5)$ as the only possibility for $Y$, see Figure \ref{Family3} (left). \\

    \item $\text{tw}(Y)\leq-3$ and $d_2=1$ \\
     We have that $e_0$ is $-1$ and the subgraph $\Gamma'$ consists of $e_0$ and the initial part of a leg (say the first one) made by a sequence of $a=d_1-1\geq1$ vertices with framing $-2$. After the blow-downs, the framing of every vertex connected to $e_0$ (except the one on the first leg, which disappears) increases by $a+1$. All the knots in the resulting surgery presentation are unknots. If $n>3$ then \[e(Y)\geq\dfrac{a}{a+1}+(n-2)\cdot\dfrac{1}{a+3}+\dfrac{1}{a+4}-1>0\:,\] thus $\Gamma$ would be indefinite; hence, we can assume that $n=3$. In addition, we observe that if there were no vertex with framing $-a-4$ connected to $e_0$ then \[e(Y)\geq\dfrac{a}{a+1}+2\cdot\dfrac{1}{a+3}-1=\dfrac{a-1}{(a+1)(a+3)}\geq0\:,\] and again this is a contradiction. A third inequality can now be written when $a\geq2$ leading to \[e(Y)\geq\dfrac{a}{a+1}+\dfrac{1}{a+3}+\dfrac{1}{a+4}-1=\dfrac{a^2+2a-5}{(a+1)(a+3)(a+4)}>0\:,\] and forcing $a=1$. If the length of the first leg is $a$, this shows that $Y=M(-1;\frac{1}{2},\frac{b+1}{4b+5},\frac{c+1}{3c+4})$ where $b,c\geq0$ are the numbers of $-2$-vertices in the second and third leg. 
     
     As in the previous case, if $c\geq2$ then $e(Y)\geq\frac{1}{2}+\frac{1}{5}+\frac{3}{10}-1=0$, thus $c=1$. Finally, if $b\geq1$ then $e(Y)\geq\frac{1}{2}+\frac{b+1}{4b+5}+\frac{2}{7}-1=\frac{2b-1}{56b+70}>0$, implying $b=0$. This leaves $\Sigma(2,5,7)$ as the only possibility for $Y$, see Figure \ref{Family3} (middle). 
     
     If the length of the first leg were bigger than $a$ then we observe that any additional vertex would have increased $e(Y)$; therefore, we just need to check that $e(Y)>0$ immediately after we add a single vertex with framing $-3$ (this becomes a $-2$ after blowing down $\Gamma$) on the first leg of the graph of $M(-1;\frac{1}{2},\frac{b+1}{4b+5},\frac{c+1}{3c+4})$. \\ 

    \item $\text{tw}(Y)\leq-3$ and $d_2>1$ \\
     We have that $e_0$ is $-1$ and the subgraph $\Gamma'$ consists of $e_0$ and the first two legs, such that $M(-1;r_1,r_2)$ is the Seifert fibration of $T_{d_2,d_1}$. After the blow-downs, the vertices connected to $e_0$ (except the ones in the first two legs) become a cable of $T_{d_2,d_1}$; all the other knots in the resulting surgery presentation are unknots. 
     
     From the proof of \cite[Proposition 5.1]{CM-negative}, the resulting framing on each $T_{d_2,d_1}$ is $\text{TB}_{\xist}(T_{d_2,d_1})-\epsilon=d_1d_2-d_1-d_2-\epsilon$ where $\epsilon=1,2$. The only torus knot for which this number is negative is the positive trefoil $T_{2,3}$; moreover, one has that $\text{TB}_{\xist}(T_{2,3})=1$ and then the standard Legendrian trefoils in the surgery presentation are once stabilised, implying that there can be only one of them.

     Summarising, we have that $M(-1;r_1,r_2)$ represents $T_{2,3}$, which means $r_1=\frac{1}{2}$ and $r_2=\frac{1}{3}$; in addition, there are three legs, and the first vertex on the third leg needs to have the precise framing so that, after the blow-downs, this will become $\text{TB}_{\xist}(T_{2,3})-2=-1$. The remaining of the third leg is untouched by the blow-downs, so all the unknots in it have framing $-2$. This leaves the family $\Sigma(2,3,6k+1)$ with $k\geq1$ as the only possibility, see Figure \ref{Family3} (right). If either of the first two legs had more than one vertex then $e(Y)>0$ which is excluded, as in the previous case. 
   \end{itemize}

\medskip
\end{document}